\documentclass[letterpaper, 10 pt, conference, onecolumn]{ieeeconf}

\IEEEoverridecommandlockouts                              % This command is only needed if 
\usepackage{dsfont}
\usepackage{amssymb}
\usepackage{eqnarray} %temp

\usepackage{amsmath}
\usepackage{mathalfa} % add for greek aclligraphic
\usepackage{pdfpages}		% for including pdf 
\usepackage{dsfont}   		%for using  \mathds{numerals}
\usepackage{algorithm}
\usepackage{algpseudocode}

\usepackage{bbm}
                                               
\usepackage{graphicx}
\usepackage[small]{caption}
\usepackage[utf8]{inputenc}
\usepackage{psfrag}
\usepackage{amsmath}
\usepackage{amssymb}
\usepackage{mathtools}

\usepackage{comment}

\usepackage{psfrag}
\usepackage{url}

\usepackage{cite}

\newcommand{\sD}{\mathcal{D}}
\newcommand{\sS}{\mathcal{S}}

\newcommand{\sT}{\mathcal{T}}

\newcommand{\R}{\mathbb{R}}
\newcommand{\E}{\mathbb{E}}
\newcommand{\Z}{\mathbb{Z}}

\newcommand{\bv}[1]{\mathbf{#1}}
\newcommand{\policy}[2]{\pi^{#2}(#1)}

\newcommand{\abs}[1]{\left\vert #1 \right\vert}

\newcommand{\KL}{\text{KL}}

\newcommand{\DKL}{\mathcal{D}_{\KL}}

\newtheorem{Remark}{Remark}

\newtheorem{Lemma}{Lemma}

\newtheorem{Assumption}{Assumption}
\newtheorem{Proposition}{Proposition}
\newtheorem{Problem}{Problem}

\title{\LARGE \bf
Optimal Decision-Making for Autonomous Agents via Data Composition \\
}

\author{\'Emiland Garrab\'e, Martina Lamberti and Giovanni Russo$^\ast$% <-this % stops a space
\thanks{$^\ast$ emails: {\tt\small \{egarrabe,giovarusso\}@unisa.it}, {\tt\small martinalamberti3@gmail.com}. E. Garrab\'e and G. Russo with the DIEM at the University of Salerno, 84084, Salerno, Italy.} 
}

\begin{document}

\maketitle
\thispagestyle{empty}
\pagestyle{empty}

\begin{abstract}
We consider the problem of designing agents able to compute optimal decisions by composing data from multiple sources to tackle tasks involving: (i) tracking a desired behavior while minimizing an agent-specific cost; (ii) satisfying safety constraints. After formulating the control problem, we show that this is convex under a suitable assumption and find the optimal solution. The effectiveness of the results, which are turned in an algorithm, is illustrated on a {\em connected cars} application via in-silico and in-vivo experiments with real vehicles and drivers. All the experiments confirm our theoretical predictions and the deployment of the algorithm on a real vehicle shows its suitability for in-car operation.
\end{abstract}

% This paper is concerned with the design of data-driven agents able to merge data from different sources to form policies. We cast this process as a data-driven, optimal control problem and, through a convexity result, give an explicit strategy to find its optimal solution. This strategy is formulated as an algorithmic procedure. Finally, we use such a procedure to implement an in-car service, which is benchmarked on a routing task.

\section{Introduction}
%A key challenge, transversal to learning and control, is that of devising mechanisms allowing autonomous decision-makers  \cite{GARRABE202281} to emulate our unique ability of adapting and re-using knowledge gathered from a variety of sources, abstracting what we know to make the appropriate decisions when solving new tasks \cite{doi:10.1126/science.aab3050,lake_ullman_tenenbaum_gershman_2017}.

We often make decisions by composing knowledge gathered from {\em others}  \cite{doi:10.1126/science.aab3050} \textcolor{black}{and} a challenge transversal to control and learning is to devise mechanisms allowing autonomous decision-makers to emulate these abilities. Systems based on sharing data \cite{sharing_eco} are  examples where agents need to make decisions based on some form of crowdsourcing \cite{russo2020crowdsourcing}. Similar mechanisms can also be useful for the data-driven control paradigm when e.g., one needs to re-use policies synthesized on plants for which data are available to solve a control task on a new plant, for which data are  scarcely available \cite{russo2020crowdsourcing,designof,https://doi.org/10.48550/arxiv.2211.05536}.

Motivated by this, we consider the problem of designing decision-making mechanisms that enable autonomous agents to compute optimal decisions by composing information from third parties to solve tasks that involve: (i) tracking a desired behavior while minimizing an agent-specific cost; (ii) satisfying  safety constraints. Our results enable computation of the optimal behavior and are turned into an algorithm. This is  experimentally validated on a {\em connected car} application.

\subsubsection*{Related works} we briefly survey a number of works related to the results and methodological framework of this paper. The design of context-aware switches between multiple datasets for autonomous agents has been recently considered in  \cite{russo2020crowdsourcing,designof}, where the design problem,  formalized as a data-driven control (DDC) problem, did not take into account safety requirements. Results in DDC include \cite{8795639,behav2,8933093}, which take a behavioral systems 
%\cite{behav_review} 
perspective, \cite{9763859}, which finds data-driven formulas towards a model-free theory of geometric control. We also recall e.g., \cite{8039204,DBLP:conf/l4dc/WabersichZ20,9109670} for results inspired from MPC, \cite{https://doi.org/10.48550/arxiv.2211.05536} that considers data-driven control policies transfer and \cite{8703172} that tackles he problem of computing data-driven minimum-energy control for linear systems. In our control problem (see Section \ref{sec:control_problem}) we formalize the tracking of a given behavior via Kullback-Leibler (KL) divergence minimization and we refer to e.g., \cite{GARRABE202281,9029512} for examples across learning and control that involve minimizing this functional. Further, the study of mechanisms enabling agents to re-use data, also arises in the design of prediction algorithms from experts \cite{books/daglib/0016248} and of learning algorithms from multiple simulators \cite{7106543}. In a yet broader context, studies in neuroscience \cite{Mountcastle97} hint that our neocortex might implement a mechanism composing the output of the cortical columns and this might be the basis of our ability to re-use knowledge.

%In the context of an agent that directly specifies its state transition probabilities, we also recall \cite{KLC}, where a value function based on a Kullback-Leibler\cite{KL_51} component and a state-based reward signal is directly minimized. Such value functions can be calculated by casting the problem as an eigenvector problem \cite{todorov_linearly_mdps}, or through graphical inference \cite{KL_as_inference}. 

%Situations where multiple sources provide data arise in e.g. the Sharing Economy \cite{sharing_eco}, Multi-Fidelity Reinforcement Learning \cite{multifidelity}, Multi-Agent learning \cite{multi_agent_rl} or data sharing platforms for robots \cite{roboearth}, and crowdsourcing mechanisms have been linked to human cognitive function\cite{1000brains}.\\
\emph{Contributions:} we consider the problem of designing agents that dynamically combine data from heterogeneous sources to fulfill tasks that involve tracking a target behavior while optimizing a cost and satisfying safety requirements expressed as box constraints. By leveraging a probabilistic framework, we formulate a data-driven optimal control problem and, for this problem, we: (i) prove convexity under a suitable condition; (ii) find the optimal solution; (iii) turn our results into an algorithm, using it to  design an intelligent parking system for connected vehicles. Validations are performed both {\em in-silico} and {\em in-vivo}, with real cars. \textcolor{black}{As such, the main purpose of this paper is twofold: (i) we introduce, and rigorously characterize our algorithm; (ii) propose a stand-alone implementation of our results, suitable for in-vivo experiments on real cars.}  In-vivo validations were performed via an hardware-in-the-loop platform allowing to \textcolor{black}{embed} real cars/drivers in the experiments. Using the platform, we deploy our algorithm on a real vehicle showing its suitability for in-car operation. All experiments confirm the effectiveness of our approach ({documented code/data for our simulations at \url{https://tinyurl.com/3ep4pknh})}.

While our results are inspired by the ones in \cite{russo2020crowdsourcing,designof}, our paper extends these in several ways. First, the results in \cite{russo2020crowdsourcing,designof} cannot consider box constraints and hence cannot tackle the control problem of this paper. Second, \textcolor{black}{even when there are no box constraints, the results in \cite{russo2020crowdsourcing,designof} only solve an approximate version of the problem considered here. That is, the results from \cite{russo2020crowdsourcing,designof} only find an approximate, non-optimal, solution of the problem considered here. As we shall see, this means that} the solutions from \cite{russo2020crowdsourcing,designof} cannot get a better cost than the one obtained with the results of this paper. Third, the algorithm obtained from the results in this paper is deployed, and validated, on a real car and this was not done in \cite{russo2020crowdsourcing,designof}. \textcolor{black}{The \emph{in-vivo} implementation is  novel.}

\subsubsection*{Notation}
sets are in {\em calligraphic} and vectors in {\bf bold}. Given the measurable space $(\mathcal{X},\mathcal{F}_x)$, with $\mathcal{X}\subseteq\R^{d}$ ($\mathcal{X}\subseteq\Z^{d}$) and $\mathcal{F}_x$ being a $\sigma$-algebra on $\mathcal{X}$, 
a random variable on $(\mathcal{X},\mathcal{F}_x)$ is denoted by $\mathbf{X}$ and its realization by $\mathbf{x}$. The \textit{probability density (resp. {mass})  function} or \textit{pdf} (\textit{pmf}) of a continuous (discrete) $\mathbf{X}$ is denoted by $p(\mathbf{x})$. The convex subset of such probability functions (pfs) is $\sD$.  The  expectation of a function $\mathbf{h}(\cdot)$ of the continuous  variable $\mathbf{X}$ is  $\E_{{p}}[\mathbf{h}(\mathbf{X})]:=\int\mathbf{h}(\mathbf{x})p(\mathbf{x})d\mathbf{x}$, where the integral (in the sense of Lebesgue) is over the support of $p(\mathbf{x})$\textcolor{black}{, which we denote by $\sS(p)$}. The joint pf of $\bv{X}_1$, $\bv{X}_2$ is $p(\mathbf{x}_1,\mathbf{x}_2)$ and the conditional pf of $\mathbf{X}_1$ given $\mathbf{X}_2$ is $p\left( \mathbf{x}_1\mid \mathbf{x}_2 \right)$. Countable sets are denoted by $\lbrace w_k \rbrace_{k_1:k_n}$, where $w_k$ is the generic element of the set and $k_1:k_n$ is the closed set of consecutive integers between $k_1$ and $k_n$. The KL divergence between $p(\mathbf{x})$ and $q(\mathbf{x})$, where $p$ is absolutely continuous w.r.t. $q$, is  $\mathcal{D}_{\KL}\left(p \mid\mid q \right):= \int_{\textcolor{black}{\sS(p)}} p \; \ln\left( {p}/{q}\right)\,d\mathbf{x}$: it is non-negative and $0$ if and only if $p(\mathbf{x}) = q(\mathbf{x})$. In the  expressions for the expectation and KL divergence, the integral is replaced by the sum if the variables are discrete. Finally: (i) we let $\mathds{1}_\mathcal{A}(\bv{x})$  denote the indicator function being equal to $1$ if $\bv{x}\in\mathcal{A}\subseteq \mathcal{X}$ and $0$ otherwise; (ii) set exclusion is instead denoted by $\setminus$.
%\GR{Specify in which sense we have the integrals (when you introduce the KL divergence and the expectation) + in the proof of Lemma 1, make sure that the integration can be swapped with differentiation + you need to define the indicator function as used in the problem statement}

\section{The Setup}

The agent seeks to craft its behavior by combining a number of sources to fulfill a task that involves tracking a target/desired behavior while maximizing an agent-specific reward over the time horizon $\sT:= 0:T$, $T>0$. The agent's state at time step $k \in \sT$ is $\bv{x}_k \in \mathcal{X}$ and the target behavior that the agent seeks to track is $p_{0:T} := p_0(\bv{x}_0)\prod_{k\in 1:T} p(\bv{x}_k\mid\bv{x}_{k-1})$. As in \cite{russo2020crowdsourcing,designof}, we design the behavior of the agent by designing its joint pf $\pi_{0:T} := \pi(\bv{x}_0,\ldots,\bv{x}_T)$ and we have:
\begin{equation}\label{eqn:agent}
    \pi_{0:T} =\pi_0(\bv{x}_0)\prod_{k\in 1:T} \pi(\bv{x}_k\mid\bv{x}_{k-1}).
\end{equation}
That is, the behavior of the agent can be designed via the pfs $\pi(\bv{x}_k\mid\bv{x}_{k-1})$, i.e., the transition probabilities. To do so, the agent has access to $S$ sources and we denote by $\pi^{(i)}(\bv{x}_k\mid\bv{x}_{k-1})$\textcolor{black}{, with support $\sS(\pi)\subseteq \mathcal{X}$,} the behavior made available by source $i$, $i\in\mathcal{S}:= 1:S$, at $k-1$. We also let $r_k(\bv{x}_k)$ be the agent's reward for being in state $\bv{x}_k$ at $k$. 

\section{Formulation of the Control Problem}\label{sec:control_problem}

Let $\alpha_k^{(1)},\ldots,\alpha_k^{(S)}$ be weights and $\boldsymbol{\alpha}_k$ be their stack. Then, the control problem we consider can be formalized as: 
\begin{Problem}\label{prob:problem_merge}
find the sequence ${\left\{\boldsymbol{\alpha}_k^\ast\right\}_{1:T}}$ solving
%\begin{equation}
    \begin{equation*}
    \begin{aligned}
    \underset{\left\{ \boldsymbol{\alpha}_k\right\}_{1:T}}{\text{min}}
    &\DKL\left(\pi_{0:T}\mid\mid p_{0:T}\right) - \sum_{k=1}^T\E_{\pi(\bv{x}_{k-1})}\left[\tilde{r}_k(\bv{X}_{k-1})\right]\\
    \mbox{s.t.} & \ \mathbb{E}_{\pi(\bv{x}_k\mid\bv{x}_{k-1})}\left[\mathds{1}_{\mathcal{X}_k}(\bv{x}_k)\right]\geq 1-\epsilon_k, \ \ \ \forall k,\\
    & \ \pi(\bv{x}_k\mid\bv{x}_{k-1}) = \sum_{i\in\sS}\alpha_k^{(i)}\policy{\bv{x}_k\mid\bv{x}_{k-1}}{(i)}, \ \ \ \forall k, \\
    & \sum_{i\in\sS}\alpha_k^{(i)} = 1, \ \ \alpha_k^{(i)} \in [0, 1],   \ \ \ \forall k. 
    \end{aligned}
    \end{equation*}
%\end{equation}
\end{Problem}
In Problem \ref{prob:problem_merge}, $\tilde{r}_k(\bv{x}_{k-1}) := \E_{\pi(\bv{x}_k\mid\bv{x}_{k-1})}\left[r_k(\bv{X}_k)\right]$ and we note  that $\E_{\pi(\bv{x}_{k-1})}\left[\tilde{r}_k(\bv{X}_{k-1})\right] = \E_{\pi(\bv{x}_{k})}\left[r_k(\bv{X}_k)\right]$ is the expected reward for the agent when the behavior in \eqref{eqn:agent} is followed.  The problem is a finite-horizon optimal control problem with the $\boldsymbol{\alpha}_k$'s \textcolor{black}{as decision variables. As we shall see, these are generated as feedback from the agent state}  (Section \ref{sec:algorithm}). We say that $\left\{\pi^\ast\left(\bv{x}_k\mid\bv{x}_{k-1}\right)\right\}_{1:T}$, with  $\pi^\ast\left(\bv{x}_k\mid\bv{x}_{k-1}\right) = \sum_{i\in\sS}\alpha_k^{(i),\ast}\policy{\bv{x}_k\mid\bv{x}_{k-1}}{(i)}$, is the optimal behavior for the agent\textcolor{black}{, obtained by linearly combining the sources via the $\boldsymbol{\alpha}_k^\ast$'s}. In the problem, the cost formalizes the fact that the agent seeks to maximize its reward, while tracking (in the KL divergence sense) the target behavior. Minimizing the KL term amounts at minimizing the discrepancy between $\pi_{0:T}$ and $p_{0:T}$. This term can also be thought as a divergence regularizer and, when $p_{0:T}$ is uniform, it becomes an entropic regularizer.
The second and third constraints formalize the fact that, at each $k$, $\pi^\ast\left(\bv{x}_k\mid\bv{x}_{k-1}\right)\in\sD$ \textcolor{black}{and that this is a convex combination of} the pfs from the sources.  The first constraint is a box constraint and models the fact that the probability that the agent behavior is, at each $k$, inside some (e.g., safety) measurable set $\mathcal{X}_k \subseteq \mathcal{X}$ is greater than some $\textcolor{black}{\epsilon}_k \ge 0$. We now make the following
\begin{Assumption}\label{asn:bounded}
the optimal cost of Problem \ref{prob:problem_merge} is bounded.
\end{Assumption}
\begin{Remark}\label{rem:bounded}
\textcolor{black}{the cost in Problem \ref{prob:problem_merge} can be recast as $\DKL\left(\pi_{0:T}\mid\mid \tilde{p}_{0:T} \right)$, where $\tilde{p}_{0:T} \propto p_{0:T}\exp{\left(\sum_{k=1}^Tr_k(\bv{x}_k)\right)}$. This means that Assumption \ref{asn:bounded} is satisfied whenever there exists some $\tilde\pi_{0:T}$ that is feasible for Problem \ref{prob:problem_merge} and that is absolutely continuous w.r.t. $\tilde{p}_{0:T}$. See also Remark \ref{rem:sources_assumption_1}.}
\end{Remark}

\begin{algorithm}[b!]
    \caption{{Pseudo-code}}\label{alg}
\begin{algorithmic}[1]
\State \textbf{Input:} time horizon $T$, target behavior $p_{0:T}$, reward $r_k(\cdot)$, sources $\pi^{(i)}(\bv{x}_k\mid\bv{x}_{k-1})$, box constraints  (optional)
\State \textbf{Output:} optimal agent behavior $\left\{\pi^\ast\left(\bv{x}_k\mid\bv{x}_{k-1}\right)\right\}_{1:T}$
\State $\hat{r}_T(\bv{x}_T)\gets0$
%\State $\left\{\boldsymbol{\alpha}_{k}^\ast\right\}_{1:N}\gets0$
\State {\bf for} {$k = T:1$} {\bf do}
    \State $\bar{r}_k(\bv{x}_k) \gets r_k(\bv{x}_k) - \hat{r}_k(\bv{x}_k)$
    \State  $\boldsymbol{\alpha}^{\ast}_{k} (\bv{x}_{k-1}) \gets \text{minimizer of the problem in \eqref{eqn:subproblem_alg}}$;
    \State $\pi^\ast(\bv{x}_k\mid\bv{x}_{k-1})\gets\sum_{i\in\mathcal{S}}\boldsymbol{\alpha}^{(i),\ast}_{k}(\bv{x}_{k-1})\pi^{(i)}(\bv{x}_k\mid\bv{x}_{k-1})$ 
    \State $\hat{r}_{k-1}(\bv{x}_{k-1}) \gets c_k(\pi^\ast(\bv{x}_k\mid\bv{x}_{k-1}))$
    %\text{minimum of the problem in \eqref{eqn:subproblem_alg}}$
    %\State $\boldsymbol{\alpha}_{k}(\bv{x}_k)\gets$solve$(\bv{x}_k,\bar{r}_k)$
\State {\bf end for}
\end{algorithmic}
\end{algorithm}

\section{Main Results}\label{sec:algorithm}

We propose an algorithm to tackle Problem \ref{prob:problem_merge}.  %\textcolor{black}{by splitting it in sub-problems that can be solved recursively}
The algorithm takes as input $T$, the target behavior, the reward, the behaviors of the sources and the box constraints of Problem \ref{prob:problem_merge} (if any). Given the inputs, \textcolor{black}{it} returns the optimal behavior for the agent. The key steps of the algorithm are given as pseudo-code in Algorithm \ref{alg}. An agent that follows Algorithm \ref{alg} computes  $\left\{\boldsymbol{\alpha}_k\right\}_{1:N}$ via backward recursion (lines $4-9$). At each $k$, the $\boldsymbol{\alpha}_k$'s are obtained as the minimizers of
\begin{equation}\label{eqn:subproblem_alg}
\begin{aligned}
    \underset{\boldsymbol{\alpha}_k}{\min}
    & \ c_k(\pi(\bv{x}_k\mid\bv{x}_{k-1}))\\
     \mbox{s.t.} & \ \mathbb{E}_{\pi(\bv{x}_k\mid\bv{x}_{k-1})}\left[\mathds{1}_{\mathcal{X}_k}(\bv{x}_k)\right]\geq 1-\epsilon_k\\
    & \ \pi(\bv{x}_k\mid\bv{x}_{k-1}) = \sum_{i\in\sS}\alpha_k^{(i)}\policy{\bv{x}_k\mid\bv{x}_{k-1}}{(i)} \\
     & \sum_{i\in\sS}\alpha_k^{(i)} = 1, \ \ \alpha_k^{(i)} \in [0, 1],
    \end{aligned}
\end{equation}
where
\begin{equation}\label{eqn:cost_step}
\begin{split}
c_k(\pi(\bv{x}_k\mid\bv{x}_{k-1})) & := \DKL\left(\pi(\bv{x}_k\mid\bv{x}_{k-1})\mid\mid p(\bv{x}_k\mid\bv{x}_{k-1})\right) \\
& - \E_{\pi(\bv{x}_k\mid\bv{x}_{k-1})}\left[\bar{r}_k(\bv{X}_{k})\right],
\end{split}
\end{equation}
with $\bar{r}_k(\cdot)$ iteratively built within the recursion (lines $5$, $8$). The weights are used (line $7$) to compute $\pi^\ast(\bv{x}_k\mid\bv{x}_{k-1})$. 
%Next, we characterize convexity of the above problem and optimality of the solution returned by Algorithm \ref{alg}.

\begin{Remark}results are stated for continuous variables (proofs for discrete variables omitted for brevity). \textcolor{black}{Note that integrals/summations in the cost are over $\sS(\pi)$}.
\end{Remark}
\begin{Remark}\label{rem:sources_assumption_1}
\textcolor{black}{following Remark \ref{rem:bounded}, the optimal cost of the problem in (\ref{eqn:subproblem_alg}) is bounded if there exists some feasible $\tilde\pi(\bv{x}_k\mid\bv{x}_{k-1})$ that is absolutely continuous w.r.t. $\tilde p(\bv{x}_k\mid\bv{x}_{k-1}) \propto p(\bv{x}_k\mid\bv{x}_{k-1})\exp\left(\bar{r}_k(\bv{x}_k)\right)$. From the design viewpoint, this can satisfied if it holds for at least one $\policy{\bv{x}_k\mid\bv{x}_{k-1}}{(i)}$.}
\end{Remark}
Finally, we make the following
%\begin{Assumption}\label{asn:DKL_sources}
%\forall k, \forall i \in \mathcal{S}, \DKL\left(\pi^{(i)}_{1:N\textcolor{black}{|0}}||p_{1:N\textcolor{black}{|0}}\right) < +\infty$
%\end{Assumption}
\begin{Assumption}\label{asn:int_def} $\forall i \in \mathcal{S}$ and $\forall \bv{x}_{k-1}\in\mathcal{X}$,  %the pf $\policy{\bv{x}_k|\bv{x}_{k-1}}{(i)}$ is Lebesgue integrable in $\mathcal{X}$ and 
there exist some constants, say $m$ and $M$, with $0 < m \le M < +\infty$, such that $m \le \policy{\bv{x}_k\mid\bv{x}_{k-1}}{(i)} \le M$, $\forall \bv{x}_k\in\sS(\pi)$.
\end{Assumption}
\begin{Remark}
\textcolor{black}{Assumption \ref{asn:bounded} is  satisfied for e.g., Gaussian distributions. As we shall see (Section \ref{sec:application}) the assumption can be fulfilled by injecting noise in the data.}
\end{Remark}
%\textcolor{black}{***A1 is mild and A2 is a standard assumption...A sufficient condition for Assumption \ref{asn:bounded} to be verified is that the reward is finite everywhere in the state space and that, for each $\bv{x}_{k-1}$, there is at least one source with a behavior that is absolutely continuous w.r.t. the target behavior. Note that if such a source also satisfies the safety constraint, the problem is feasible everywhere in the state space, and thus recursively feasible.}
\subsection{Properties of Algorithm \ref{alg}}\label{sec:properties}

{\color{black} Before characterizing convexity of the problems recursively solved in Algorithm \ref{alg} and optimality, we give a condition ensuring feasibility of the problem in \eqref{eqn:subproblem_alg}.
\begin{Lemma}\label{lem:feasibility}
the problem in \eqref{eqn:subproblem_alg} is feasible if and only if there exists at least one source, say $\pi^{(j)}(\bv{x}_k\mid\bv{x}_{k-1})$, such that $\mathbb{E}_{\pi^{(j)}(\bv{x}_k\mid\bv{x}_{k-1})}\left[\mathds{1}_{\mathcal{X}_k}(\bv{x}_k)\right]\geq 1-\epsilon_k$.
\end{Lemma}
\begin{proof} the {\em if} part clearly holds. For the {\em only if} part we prove that if problem (\ref{eqn:subproblem_alg}) is infeasible, then  $\max_i\E_{\pi^{(i)}(\bv{x}_k\mid\bv{x}_{k-1})}\left[\mathds{1}_{\mathcal{X}_k}(\bv{x}_k)\right]< 1-\epsilon_k$. In fact, if the problem is infeasible, then for all $\boldsymbol{\alpha}_k$ such that $\sum_{i\in\sS}\alpha_k^{(i)} = 1$ and $\alpha_k^{(i)} \in [0, 1]$ it must hold that
\begin{equation*}    
\int_{\mathcal{X}_k}\sum_{i\in\sS}\alpha_k^{(i)}\policy{\bv{x}_k\mid\bv{x}_{k-1}}{(i)}d\bv{x}_k < 1-\epsilon_k.
\end{equation*}

Note that this must also hold for all $\boldsymbol{\alpha}_k$ such that $\sum_{i\in\sS}\alpha_k^{(i)} = 1$ and $\alpha_k^{(i)} \in \{0, 1\}$, as these are contained in the set of real-valued $\boldsymbol{\alpha}_k$'s. We conclude the proof by noticing that, if $\boldsymbol{\alpha}_k$ is such that $\alpha_k^{(i)}=0 \forall i\neq j$ and $\alpha_k^{(j)}=1$, then 
\begin{equation*}
\begin{aligned}
\int_{\mathcal{X}_k}\sum_{i\in\sS}\alpha_k^{(i)}\policy{\bv{x}_k\mid\bv{x}_{k-1}}{(i)}d\bv{x}_k &= \int_{\mathcal{X}_k}\policy{\bv{x}_k\mid\bv{x}_{k-1}}{(j)}d\bv{x}_k\\
&= E_{\pi^{(j)}(\bv{x}_k\mid\bv{x}_{k-1})}\left[\mathds{1}_{\mathcal{X}_k}(\bv{x}_k)\right].
\end{aligned}
\end{equation*}
It then follows that, $\forall j$, 
\begin{equation*}
E_{\pi^{(j)}(\bv{x}_k\mid\bv{x}_{k-1})}\left[\mathds{1}_{\mathcal{X}_k}(\bv{x}_k)\right]< 1-\epsilon_k.
\end{equation*}

\end{proof}}

\subsubsection*{Convexity}
we are now ready to prove the following
\begin{Proposition}\label{lem:convexity}
let Assumption \ref{asn:int_def} hold. Then, the problem in \eqref{eqn:subproblem_alg} is convex.
\end{Proposition}

\begin{proof}
%we start with showing convexity of the constraints set. 
Clearly, the second and third constraint in the problem are convex. For the first constraint, we get 
\begin{equation*}
    \begin{aligned}
\mathbb{E}_{\pi(\bv{x}_k\mid\bv{x}_{k-1})}\left[\mathds{1}_{\mathcal{X}_k}(\bv{x}_k)\right] &= \int  \pi(\bv{x}_k\mid\bv{x}_{k-1})\mathds{1}_{\mathcal{X}_k}(\bv{x}_k)d\bv{x}_k \\ &= \int_{\mathcal{X}_k}\sum_{i \in \mathcal S}\alpha_k^{(i)}\pi^{(i)}(\bv{x}_k\mid\bv{x}_{k-1})d\bv{x}_k \\ &= \sum_{i \in \mathcal S}\alpha_k^{(i)}\int_{\mathcal{X}_k}\pi^{(i)}(\bv{x}_k\mid\bv{x}_{k-1})d\bv{x}_k,
 \end{aligned}
\end{equation*}
which is therefore convex in the decision variables. \\
Now, we show that the cost is also convex in these variables and we do so by explicitly computing, for each $\bv{x}_{k-1}$, its Hessian, say $\bv{H}(\bv{x}_{k-1})$. Specifically, after embedding the second constraint of the problem in (\ref{eqn:subproblem_alg}) in the cost and differentiating with respect to the decision variables we get, for each $j\in\sS$:
    \begin{equation*}
    \begin{split}
       & \frac{\partial c_k}{\partial \alpha_k^{(j)}}  := \frac{\partial}{\partial \alpha_k^{(j)}}\int_{\textcolor{black}{\sS(\pi)}} \sum_{i\in\sS}\alpha_k^{(i)}\policy{\bv{x}_k\mid\bv{x}_{k-1}}{(i)}\left(\log\left(\frac{\sum_{i\in\sS}\alpha_k^{(i)}\policy{\bv{x}_k\mid\bv{x}_{k-1}}{(i)}}{p(\bv{x}_k\mid\bv{x}_{k-1})}\right) - \bar{r}_k(\bv{x}_{k})\right)d\bv{x}_k\\
        &= \int_{\textcolor{black}{\sS(\pi)}} \frac{\partial}{\partial \alpha_k^{(j)}} \sum_{i\in\sS}\alpha_k^{(i)}\policy{\bv{x}_k\mid\bv{x}_{k-1}}{(i)}\left(\log\left(\frac{\sum_{i\in\sS}\alpha_k^{(i)}\policy{\bv{x}_k\mid\bv{x}_{k-1}}{(i)}}{p(\bv{x}_k\mid\bv{x}_{k-1})}\right) - \bar{r}_k(\bv{x}_{k})\right)d\bv{x}_k\\
        &= \int_{\textcolor{black}{\sS(\pi)}} \policy{\bv{x}_k\mid\bv{x}_{k-1}}{(j)}\Bigg( \Bigg.\log\left(\sum_{i\in\sS}\alpha_k^{(i)}\policy{\bv{x}_k\mid\bv{x}_{k-1}}{(i)}\right) - \log\left(p(\bv{x}_k\mid\bv{x}_{k-1})\right) - \bar{r}_k(\bv{x}_k) + 1\Bigg. \Bigg) d\bv{x}_k . \\
    \end{split}
    \end{equation*}
The above chain of identities was obtained by swapping integration and differentiation, leveraging the fact that the cost is smooth in the decision variables. Similarly, we  get 

\begin{equation*}
\begin{split}
\frac{\partial^2 c_k}{\partial \alpha_k^{(j)^2}} &=  \frac{\partial}{\partial \alpha_k^{(j)}} \int_{\textcolor{black}{\sS(\pi)}} \policy{\bv{x}_k\mid\bv{x}_{k-1}}{(j)}\Bigg( \Bigg.\log\left(\sum_{i\in\sS}\alpha_k^{(i)}\policy{\bv{x}_k\mid\bv{x}_{k-1}}{(i)}\right) - \log\left(p(\bv{x}_k\mid\bv{x}_{k-1})\right) - \bar{r}_k(\bv{x}_k) + 1\Bigg. \Bigg) d\bv{x}_k \\
&= \textcolor{black}{\int_{\sS(\pi)} \frac{\policy{\bv{x}_k\mid\bv{x}_{k-1}}{(j)}^2}{\sum_{i\in\sS}\alpha_k^{(i)}\policy{\bv{x}_k\mid\bv{x}_{k-1}}{(i)}}d\bv{x}_k}\\
&=: h_{\textcolor{black}{jj}}(\bv{x}_{k-1}), 
\end{split}
\end{equation*}
    
    and, for each $m \ne j$, $m\in\sS$,

\begin{equation*}
\frac{\partial^2 c_k}{\partial \alpha_k^{(j)}\partial \alpha_k^{(m)}} = \int_{\textcolor{black}{\sS(\pi)}} \frac{\policy{\bv{x}_k\mid\bv{x}_{k-1}}{(j)}\policy{\bv{x}_k\mid\bv{x}_{k-1}}{(m)}}{\sum_{i\in\sS}\alpha_k^{(i)}\policy{\bv{x}_k\mid\bv{x}_{k-1}}{(i)}}d\bv{x}_k =: h_{\textcolor{black}{j}m}(\bv{x}_{k-1}).
\end{equation*}
Also, following Assumption \ref{asn:int_def}, $\forall j,m\in \sS$ we have that
\begin{equation*}
\int_{\textcolor{black}{\sS(\pi)}}\abs{ \frac{\policy{\bv{x}_k\mid\bv{x}_{k-1}}{(j)}\policy{\bv{x}_k\mid\bv{x}_{k-1}}{(m)}}{\sum_{i\in\sS}\alpha_k^{(i)}\policy{\bv{x}_k\mid\bv{x}_{k-1}} {(i)}}}d\bv{x}_k\le \frac{M}{m}
\end{equation*}
where we used the third constraint. That is, the above integrals are well defined and thus we can conclude the proof by computing $\bv{v}^T\bv{H}(\bv{x}_{k-1})\bv{v}$ for some non-zero $\bv{v}\in\mathbb{R}^S$:
\begin{align*}
    \bv{v}^T\bv{H}(\bv{x}_{k-1})\bv{v} &= \sum_{j,m}v_jv_mh_{jm} (\bv{x}_{k-1})\\
    &= \sum_{j,m}v_jv_m\int_{\textcolor{black}{\sS(\pi)}}a_{jm}(\bv{x}_k,\bv{x}_{k-1})d\bv{x}_k\\
    &= \int_{\textcolor{black}{\sS(\pi)}}\sum_{j,m}v_jv_ma_{jm}(\bv{x}_k,\bv{x}_{k-1})d\bv{x}_k,
\end{align*}
where the $a_{jm}$'s are the elements of the matrix
\begin{equation*}
A(\bv{x}_k,\bv{x}_{k-1}) := \bar{\pi} (\bv{x}_{k},\bv{x}_{k-1})\left[\begin{array}{*{20}c}
\policy{\bv{x}_k\mid\bv{x}_{k-1}}{(1)}\\ \vdots\\ \policy{\bv{x}_k\mid\bv{x}_{k-1}}{(S)}
\end{array}\right] \cdot \\
 \cdot\left[\begin{array}{*{20}c}
\policy{\bv{x}_k\mid\bv{x}_{k-1}}{(1)}  \ldots  \policy{\bv{x}_k\mid\bv{x}_{k-1}}{(S)}
\end{array}\right],
\end{equation*}
with $\bar{\pi} (\bv{x}_{k},\bv{x}_{k-1}) := {1}/\left({\sum_{i\in\sS}\alpha_k^{(i)}\policy{\bv{x}_k\mid\bv{x}_{k-1}}{(i)}}\right)$. The above expression is indeed positive semi-definite for each $\bv{x}_k$, $\bv{x}_{k-1}$ and we can draw the desired conclusion. \end{proof}

\subsubsection*{Optimality} we can now prove the following

\begin{Proposition}\label{thm:1}
let Assumption \ref{asn:int_def} and Assumption \ref{asn:bounded} hold. Then, Algorithm \ref{alg} gives an optimal solution for Problem \ref{prob:problem_merge}.    
\end{Proposition}
\begin{proof}
the chain rule for the KL divergence and the linearity of expectation imply that the cost can be written as
\begin{equation}\label{eqn:cost_split}
 \DKL\left(\pi_{0:T-1}\mid\mid p_{0:T-1}\right)  - \sum_{k=1}^{T-1}\E_{\pi(\bv{x}_{k-1})}\left[\tilde{r}_k(\bv{X}_{k-1})\right] + \E_{\pi(\bv{x}_{T-1})}\left[c_T(\pi(\bv{x}_T\mid\bv{x}_{T-1}))\right],
\end{equation}
where $c_T(\pi(\bv{x}_T\mid\bv{x}_{T-1}))$ is defined as in \eqref{eqn:cost_step} with $\bar{r}_T(\bv{x}_T)$ given by Algorithm \ref{alg} -- see lines $3$ and $5$ and note that, at time step $T$, $\bar{r}_T(\bv{x}_T) = {r}_T(\bv{x}_T)$. To obtain the above expression, the fact that $c_T(\pi(\bv{x}_T\mid\bv{x}_{T-1}))$ only depends on $\bv{x}_{T-1}$ was also used. Hence, Problem \ref{prob:problem_merge} can be split into the sum of two sub-problems: a first problem over $k\in 0:T-1$ and the second for $k=T$. For this last time step, the problem can be solved independently on the others and is given by:
\begin{equation}\label{eqn:subproblem_N}
    \begin{aligned}
    \underset{\boldsymbol{\alpha}_T}{\text{ min }}
    & \ \E_{\pi(\bv{x}_{T-1})}[c_T(\pi(\bv{x}_T\mid\bv{x}_{T-1}))]\\
    \mbox{s.t.} & \ \mathbb{E}_{\pi(\bv{x}_T\mid\bv{x}_{T-1})}\left[\mathds{1}_{\mathcal{X}_T}(\bv{x}_T)\right]\geq 1-\epsilon_T,\\
    & \ \pi(\bv{x}_T\mid\bv{x}_{T-1}) = \sum_{i\in\sS}\alpha_T^{(i)}\policy{\bv{x}_T\mid\bv{x}_{T-1}}{(i)}, \\
    & \sum_{i\in\sS}\alpha_T^{(i)} = 1, \ \ \alpha_T^{(i)} \in [0, 1].
    \end{aligned}
\end{equation}
Using linearity of the expectation and the fact that the decision variable is independent on $\pi(\bv{x}_{T-1})$, we have that the minimizer of the problem in \eqref{eqn:subproblem_N} is the same as the problem in \eqref{eqn:subproblem_alg} with $k=T$. Following Proposition \ref{lem:convexity}, such a problem is convex and we denote its optimal solution as $\boldsymbol{\alpha}^{\ast}_{T} (\bv{x}_{T-1})$ -- see line $6$ of Algorithm \ref{alg} -- and the optimal cost of the problem, which is bounded by Assumption \ref{asn:bounded},  is $c_T(\pi^{\ast}(\bv{x}_T\mid\bv{x}_{T-1}))$, where: $$\pi^{\ast}(\bv{x}_T\mid\bv{x}_{T-1}) = \sum_{i\in\mathcal{S}}\boldsymbol{\alpha}^{(i),\ast}_{T}(\bv{x}_{T-1})\pi^{(i)}(\bv{x}_T\mid\bv{x}_{T-1})\textcolor{black}{.}$$ This gives $\hat{r}_{T-1}(\bv{x}_{T-1})$ in Algorithm \ref{alg} (lines $7-8$), thus yielding the steps for the backward recursion of the Algorithm \ref{alg} at time step $T$. Now, the minimum value of the problem in \eqref{eqn:subproblem_N} is given by $\E_{\pi(\bv{x}_{T-1})}\left[\hat{r}_{T-1}(\bv{X}_{T-1})\right]$. Hence, the cost of Problem \ref{prob:problem_merge} becomes 

\begin{equation*}
\DKL\left(\pi_{0:T-1}\mid\mid p_{0:T-1}\right) - \sum_{k=1}^{T-1}\E_{\pi(\bv{x}_{k-1})}\left[\tilde{r}_k(\bv{X}_{k-1})\right]  + \E_{\pi(\bv{x}_{T-1})}\left[\hat{r}_{T-1}(\bv{X}_{T-1})\right]. 
\end{equation*}

Then, following the same reasoning used to obtain \eqref{eqn:cost_split} and by noticing that 

\begin{equation*}
\E_{\pi(\bv{x}_{T-1})}\left[\hat{r}_{T-1}(\bv{X}_{T-1})\right] = \E_{\pi(\bv{x}_{T-2})}\left[\E_{\pi(\bv{x}_{T-1}\mid\bv{x}_{T-2})}\left[\hat{r}_{T-1}(\bv{X}_{T-1})\right]\right], 
\end{equation*}
we get:
%(see \cite{designof} for the details):
\begin{equation*}
\DKL\left(\pi_{0:T-2}\mid\mid p_{0:T-2}\right)  - \sum_{k=1}^{T-2}\E_{\pi(\bv{x}_{k-1})}\left[\tilde{r}_k(\bv{X}_{k-1})\right]  + \E_{\pi(\bv{x}_{T-2})}\left[c_{T-1}(\pi(\bv{x}_{T-1}\mid\bv{x}_{T-2}))\right],
\end{equation*}
where $c_{T-1}(\pi(\bv{x}_{T-1}\mid\bv{x}_{T-2}))$ is again given in \eqref{eqn:cost_step} with $\bar{r}_{T-1}(\bv{x}_{T-1})$ again defined as in Algorithm \ref{alg}. 
By iterating the arguments above, we find that at each time step Problem \ref{prob:problem_merge} can always be split as the sum of two sub-problems, where the last sub-problem can be solved independently on the previous ones. Moreover, the minimizer of this last sub-problem is always the solution of a problem of the form
\begin{equation}\label{eqn:subproblem}
    \begin{aligned}
    \underset{\boldsymbol{\alpha}_k}{\text{min}}
    &\DKL\left(\pi(\bv{x}_k\mid\bv{x}_{k-1})\mid\mid p(\bv{x}_k\mid\bv{x}_{k-1})\right) - \E_{\pi(\bv{x}_k\mid\bv{x}_{k-1})}\left[\bar{r}_k(\bv{X}_{k})\right]\\
    \mbox{s.t.} & \ \mathbb{E}_{\pi(\bv{x}_k\mid\bv{x}_{k-1})}\left[\mathds{1}_{\mathcal{X}_k}(\bv{x}_k)\right]\geq 1-\epsilon_k,\\
    & \ \pi(\bv{x}_k\mid\bv{x}_{k-1}) = \sum_{i\in\sS}\alpha_k^{(i)}\policy{\bv{x}_k\mid\bv{x}_{k-1}}{(i)}, \\
    & \sum_{i\in\sS}\alpha_k^{(i)} = 1, \ \ \alpha_k^{(i)} \in [0, 1],
    \end{aligned}
\end{equation}
where $\bar{r}_k(\bv{x}_k) := r_k(\bv{x}_k) - \hat{r}_k(\bv{x}_k)$, with $\hat{r}_k(\bv{x}_k) := c_{k+1}({\pi}^\ast(\bv{x}_{k+1}\mid\bv{x}_{k}))$ and $\pi^\ast(\bv{x}_k\mid\bv{x}_{k-1})=\sum_{i\in\mathcal{S}}\boldsymbol{\alpha}^{(i),\ast}_{k}(\bv{x}_{k-1})\pi^{(i)}(\bv{x}_k\mid\bv{x}_{k-1})$. This yields the desired conclusions.
\end{proof}
\begin{Remark}\label{rem:comparison}
\textcolor{black}{the results from \cite{russo2020crowdsourcing,designof}  solve an approximate version of Problem \ref{prob:problem_merge} when there are no box constraints. Hence, even in this special case, the results from \cite{russo2020crowdsourcing,designof} do not lead to the optimal  solution found with Algorithm \ref{alg}. Specifically, the approximate solution from \cite{russo2020crowdsourcing,designof} corresponds to the optimal solution of a problem with additional binary constraints on the decision variables. As a result, the algorithm from \cite{russo2020crowdsourcing,designof} searches solutions in a  feasibility domain that is contained in the feasibility domain of Problem \ref{prob:problem_merge}. Hence, the solutions found in \cite{russo2020crowdsourcing,designof} cannot achieve a better cost than the ones obtained via  Algorithm \ref{alg}.}
\end{Remark}
%%% BETTER FINE TUNE THE REPLY TO THE REVIEWER SAYING THAT WE CANNOT ADD THIS EXPLICIT COMMENT IN THE PAPER DUE TO PAGE CONSTRAINTS
%\begin{Remark}
%\textcolor{black}{Proposition \ref{lem:convexity} gives a condition for the convexity of the problems recursively solved within Algorithm \ref{alg}, while Proposition \ref{thm:1} show that Problem \ref{prob:problem_merge} can be indeed split into the problems in (\ref{eqn:subproblem_alg}). The proof of Proposition \ref{thm:1} leverages the convexity result.}
%\end{Remark}

\section{Designing an Intelligent Parking System}\label{sec:application}

We now use Algorithm \ref{alg} to design an intelligent parking system for {connected cars} and validate the results via in-silico  and in-vivo experiments. For the latter set of experiments, Algorithm \ref{alg} is deployed on a real car and  validation is performed via an hardware-in-the-loop (HiL) platform inspired from \cite{Griggs2019}. Before reporting the results, we describe the validation scenarios and the experimental set-up. Code, maps and parameters with instructions to replicate the simulations are given at \url{https://tinyurl.com/3ep4pknh}. 

\begin{figure}[b!]
    \centering
    \includegraphics[width=0.6\columnwidth]{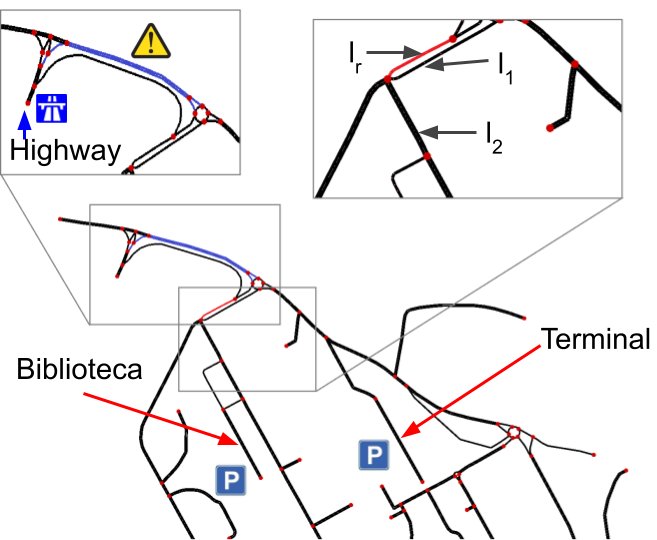}
    \caption{campus map. The magnified areas show the obstructed road link (in blue) and links used within the validations. Colors online.}
    \label{fig:unisa_map}
\end{figure}

\subsection{Validation Scenarios and Experimental Set-up}

We consider the problem of routing vehicles in a given geographic area to find parking. In all experiments we consider a morning rush scenario at the University of Salerno campus (see Figure \ref{fig:unisa_map}). Specifically, cars arrive to the campus through a highway exit and, from here, users seek to park in one of the parking locations: {\em Biblioteca} and {\em Terminal}.

In this context, vehicles are agents equipped with Algorithm \ref{alg}. The set of road links within the campus is $\mathcal{X}$ and time steps are associated to the time instants when the vehicle changes road link. The state of the agent, $x_k$, is the road link occupied by the vehicle  at time step $k$. Given this set-up, at each $k$ Algorithm \ref{alg} outputs the turning probability for the car given the current car link,  $\pi^{\ast}({x}_k\mid {x}_{k-1})$. The next direction for the car is then obtained by sampling from this pf. Agents have access to a set of sources, each providing different routes.  As discussed in \cite{crawling}, sources might be third parties navigation services, routes collected from other cars/users participating to some sharing service. Agents wish to track their target/desired behavior (driving them to the preferred parking -- Terminal in our experiments) and the reward depends on the actual road conditions within the campus. Links adjacent to a parking lot are assigned a constant reward of: (i) $3.8$ if the parking has available spaces; (ii) {0 when it becomes full}. {Unparked cars already on full parking lots are assigned a target behavior leading them to another parking.} In normal road conditions, the reward for the other links is $0$ and becomes $-20$ when there is an obstruction. \textcolor{black}{In our experiments, the reward was selected heuristically so that it would be: (i) sufficiently penalizing for links affected by obstruction; (ii) encouraging cars to drive towards parking lots with available spaces.} In the first scenario (Scenario $1$) there are no box constraints: this is done to benchmark Algorithm \ref{alg} with \cite{russo2020crowdsourcing,designof}. To this aim, we use the campus map from \cite{crawling} in which \cite{russo2020crowdsourcing,designof} were thoroughly validated via simulations. Then, we show that by introducing box constraints Algorithm \ref{alg} can effectively regulate access of vehicles through selected road links. This is Scenario $2$ and we considered three situations: (A) the road towards the {\em Biblioteca} parking lot is forbidden. To account for this, we set in Algorithm \ref{alg} $\mathcal{X}_k = \mathcal{X}\setminus {{l_2}}$, where the link ${l_2}$ is shown in Figure \ref{fig:unisa_map}, and $\epsilon_k = 0.027$;  (B) the set $\mathcal{X}_k$ as before but now $\epsilon_k = 0.5$; (C) the road towards the {\em Terminal} parking lot is forbidden and in this case we set $\mathcal{X}_k = \mathcal{X}\setminus {l_1}$, $\epsilon_k = 0.027$ (see Figure \ref{fig:unisa_map} for link ${l_1}$). For this last scenario, Algorithm \ref{alg} is validated both in-silico and in-vivo. Next, we describe the simulation and HiL experimental set-ups.

\subsubsection*{Simulation set-up}  simulations were performed in SUMO \cite{SUMO2018}; {see also \cite{crawling} for a description of the pipeline to import maps and traffic demands.} In our simulations, each parking lot can accommodate up to $50$ cars and we  generated the traffic demand so that $100$ cars would arrive on campus at $5$-second intervals. All the cars seek to park and, by doing so, the parking capacity is saturated. Given this setting, we simulated a road obstruction on the main link (in blue in Figure \ref{fig:unisa_map}) from the highway exit to the campus entrance. This was done by restricting, in SUMO, the speed of the link to less than one kilometer per hour. Information on the cars in the simulation are contained in the stand-alone file {\tt agent.npy}. Instead, the pfs associated with the sources (described below)  are all stored in {\tt behaviors.npy}. 

\subsubsection*{HiL set-up} the platform embeds a real vehicle into a SUMO simulation. By doing so, performance of the algorithm deployed on a real car can be assessed under arbitrary synthetic traffic conditions generated via SUMO. The high-level architecture of the platform is shown in Figure \ref{fig:HIL_archi}. The platform creates an avatar of the real car in the SUMO simulation. Namely, as shown in Figure \ref{fig:HIL_archi}, the position of the real car is embedded in SUMO by using a standard smartphone to collect its GPS coordinates. These coordinates are then sent via bluetooth to a computer, also hosted on the car in the current  implementation. The connection is established via an off-the-shelf app, which writes the coordinates in a {\tt rfcomm} file. By using the {\tt pySerial} library, an interface was developed to read data in Python. Here, a script was designed leveraging {\tt pynmeaGPS} to translate the data in the {\tt NMEA} format for longitude/latitude coordinates. With this data format, a Python script was created to place the avatar of the real car in the position given by the coordinates. A GUI is also included to highlight the trajectory of the real car on the map and an off-the-shelf text-to-speech routine is used to provide audio feedback to the driver on the vehicle.
%\begin{Remark}
%For both set-ups, all interfaces between SUMO and our algorithm are handled through TraCI in the main Python file. Logs are stored as {\tt .npy} files.
%\end{Remark}

\begin{figure}[b!]
    \centering
    \includegraphics[width=0.6\linewidth]{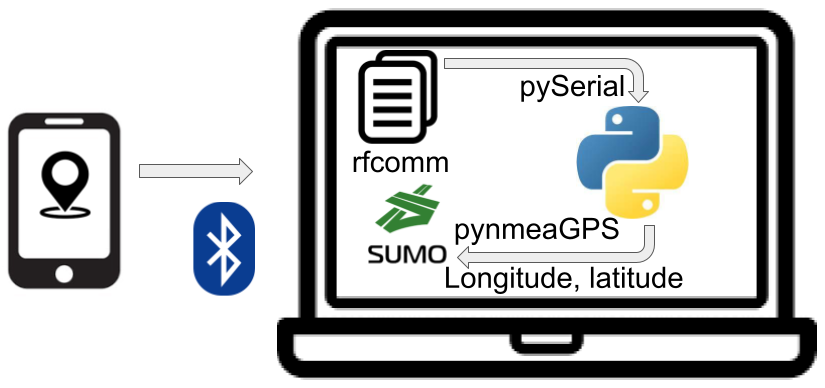}
    \caption{HiL functional architecture.}
    %The smartphone acts as a GPS sensor, while the NMEA data is read and converted in Python.}
    \label{fig:HIL_archi}
\end{figure}

\subsection{In-car Implementation of the Algorithm}\label{sec:incar}
For both the in-silico and in-vivo experiments, Algorithm \ref{alg} was implemented in Python as a stand-alone class so that each car equipped with the algorithm could function as a stand-alone agent. The class has methods accepting all the input parameters of Algorithm \ref{alg} and providing as output the car behavior computed by the algorithm. The optimization problems within the algorithm loop were solved via the Python library {\tt scipy.optimize}. Additionally, the class also implements  methods to compute the cost and to support receding horizon implementations of Algorithm \ref{alg}. In our experiments, \textcolor{black}{we used this receding horizon implementation:} the width of the receding horizon window was $T=5$ and every time the car entered in a new link/state computations were triggered. \textcolor{black}{The pfs from the sources in our experiments were such that, at each $k$, feasibility of the problem was ensured (see Lemma \ref{lem:feasibility})}. Following \cite{crawling}, we also implemented an utility function that restricts calculations of the agent only to the road links that can be reached in $T$ time steps (rather than through the whole state space/map). With this feature, \textcolor{black}{in our experiments the algorithm  took on average approximately half a second to output a behavior, less than the typical time taken to drive through a road link}.\footnote{this duration was measured between the moment where the GUI shows the car merging on a new link and the moment where new directions are displayed. The simulation was run on a modern laptop.} Finally, the pfs of the sources were obtained via built-in routing functions in SUMO and we added noise to the routes so that Assumption \ref{asn:int_def} would be fulfilled \textcolor{black}{(for each road link, $\sS(\pi)$ is the set of outgoing links).} See our github for the details.

%a set of target destinations are selected (in this work, we chose both parking lots and the south of the campus). For a given road link and source, the next move towards the corresponding destination is assigned a probability of $0.94$, while the other accessible road links share the remaining probability. Such stochasticity accounts for driver error and for privacy requirements, with users being able to corrupt their historical data with e.g. laplacian noise. Stochastic behaviors can also arise when several actions fulfil the user's goal.\\
%Finally, if a connected car is on a parking lot with no vacant spaces (this happens if the parking lot became full after the car engaged on the corresponding road link but before it was able to reach a parking space), it is assigned a new target behavior and queries the decision-maker again.\\
%This, and all utility functions (such as computing the sufficient state space for the decision-making process, assigning parking spaces to cars when possible, or logging results), is done in the main Python file.

\subsection{Experimental Results}

%We now report the results from our experiments.

\subsubsection*{Simulation results}
first, we benchmarked the performance obtained by Algorithm \ref{alg}  against these from the algorithm in \cite{russo2020crowdsourcing,designof}, termed as crowdsourcing algorithm in what follows. To this aim, we considered Scenario $1$ and performed two sets of $10$ simulations. In the first set of experiments, Algorithm \ref{alg} was used to determine the behavior of cars on campus (note that Assumption \ref{asn:bounded} is fulfilled). In the second set of simulations, the cars instead used the crodwourcing algorithm. Across the simulations, we recorded the number of cars that the algorithms were able to park. The results are illustrated in Figure \ref{fig:unparked_proba_plot} \textcolor{black}{(top panel)}. The figure shows that the crowdsourcing algorithm was not able to park all the cars within the simulation. This was instead achieved by Algorithm \ref{alg}, which outperformed the algorithm from \cite{russo2020crowdsourcing,designof}. To further quantify the performance, we also computed the average time spent by a car looking for a parking space after it enters the simulation (ATTP: average time-to-parking). Across the simulations, the ATTP for the algorithm in \cite{russo2020crowdsourcing} was of $224.74 \pm 19.67$, while for Algorithm \ref{alg} it was of $151.32 \pm 30.59$ (first quantities are means, second quantities are  standard deviations). That is, Algorithm \ref{alg} yielded an average improvement of $32.7\%$ in the ATTP. Then, we simulated the three cases of Scenario $2$ to verify that Algorithm \ref{alg} can  effectively regulate access through specific links. The constraints for the three cases of Scenario $2$ were given as an input to the algorithm and in Figure \ref{fig:unparked_proba_plot} \textcolor{black}{(bottom panel)} the optimal solution $\pi^{\ast}({x}_k\mid x_{k-1}=l_r)$ is shown. The figure shows that the optimal solution indeed fulfills the constraints (the road link $l_r$ is  in Figure \ref{fig:unisa_map}).

\begin{figure}[b!]
    \centering
    \includegraphics[width=0.6\columnwidth]{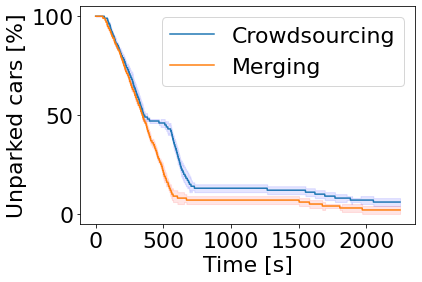} 
    \includegraphics[width=0.6\columnwidth]{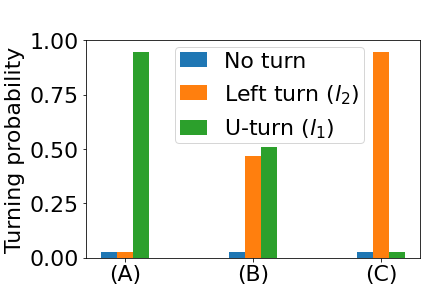}
    \caption{Top: unparked cars over time for crowdsourcing and Algorithm \ref{alg}. Solid lines are means across the simulations, shaded areas are confidence intervals (one standard deviation). Bottom: $\pi^{\ast}({x}_k\mid x_{k-1}=l_r)$  for the three cases of Scenario $2$. The pfs satisfy the constraints. Link definitions in Figure \ref{fig:unisa_map}.}
    \label{fig:unparked_proba_plot}
\end{figure}

\subsubsection*{HiL results} we deployed Algorithm \ref{alg} on a real vehicle using the HiL platform and validated its effectiveness in Scenario $2$ (C): the target behavior of the agent would make the real car reach the {\em Terminal} parking but this route is forbidden. What we observed in the experiment was that, once the car entered in the campus, this was re-routed towards the {\em Biblioteca} parking. The re-routing was an effect of Algorithm \ref{alg} computing the rightmost pf in Figure \ref{fig:unparked_proba_plot} \textcolor{black}{(bottom panel)}. A video of the HiL experiment is available on our github. The video shows that the algorithm is suitable for real car operation: it would run smoothly during the drive, providing feedback to the driver on the vehicle. Figure \ref{fig:HIL} shows the car's route recorded during the experiment.

\begin{figure}[b!]
    \centering
    \includegraphics[width=0.5\columnwidth]{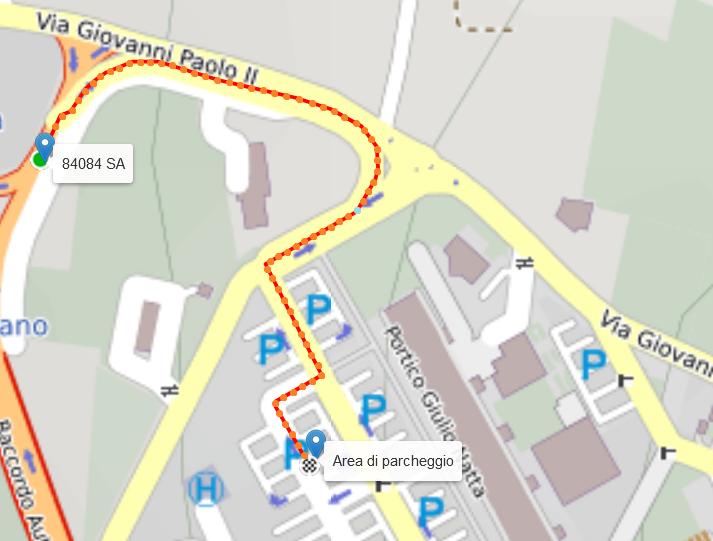}\\
    \caption{Route of the real vehicle. The continuous line shows the  GPS position during the HiL experiment (map from OpenStreetMaps).}
    \label{fig:HIL}
\end{figure}

\section{Conclusions}

We considered the problem of designing agents able to compute optimal decisions by re-using data from multiple sources to solve tasks involving: (i) tracking a desired behavior while minimizing an agent-specific cost; (ii) satisfying certain safety  constraints. After formulating the control problem, we showed that this is convex under a mild condition and computed the optimal solution. We turned the results in an algorithm and used it to design an intelligent parking system. We  evaluated the algorithm via in-silico and in-vivo experiments with real vehicles/drivers. All experiments confirmed the effectiveness of the algorithm and its suitability for in-car operation. Besides \textcolor{black}{considering the use of other divergences in the cost  and deploying our results in more complex urban scenarios that would  use data from pedestrians and sensors on-board vehicles}, our future research will involve devising mechanisms for the composition of policies for the tasks with actuation constraints in \cite{gagliardi2020probabilistic}. 

\bibliographystyle{IEEEtran} 
\bibliography{bibliography}

\section*{Appendix}
%\noindent\textbf{Computation time}\\
We report here an investigation of how the time taken by Algorithm \ref{alg} changes as a function of the number of sources, $S$, and time horizon $T$.
This time is a crucial aspect to investigate whether the approach we propose would scale to more complex urban scenarios than the one presented in Section \ref{sec:application}, which would e.g., include more parking locations (note that these are seen as links by Algorithm \ref{alg}) and more complex road networks together with more sources that the agent could use to determine its optimal behavior. The time Algorithm \ref{alg} takes to output a behavior depends on the number of available sources, the time horizon $T$ and the number of links accessible to the car within the time horizon. We analyze the computation time w.r.t. the\\
To investigate Algorithm \ref{alg}'s computation time, we considered the same implementation as in Section  \ref{sec:incar} and ran the algorithm by varying the receding horizon window between $0$ and $5$ time steps, and the number of sources available to the agent between $1$ and $6$. For this, additional sources were taken from \cite{crawling}, where more sources were used to implement the algorithm from \cite{designof}. For each pair of these parameters, we measured the time taken by Algorithm \ref{alg} to output a behavior, by running the algorithm over each link in the network on a standard computer and taking the average of such times. This gives a fair estimate of the average runtime, as the amount of states considered depends on the density of the connections in the neighborhood of each link. 

The results of this numerical investigation are shown in Figure \ref{fig:3D}. The figure shows that the highest computation time is about one second, which appears to be suitable for our reference application. 

\begin{figure}
    \centering
    \includegraphics[width=0.5\columnwidth]{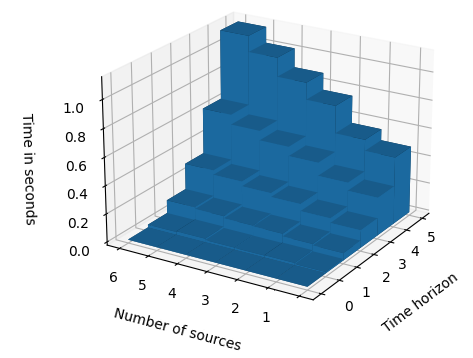}
    \caption{Computation time as a function of time horizon and number of sources}
    \label{fig:3D}
\end{figure}

%\noindent \textbf{Reward signal tuning}\\
%Correctly tuning the reward signal is a key step to ensure that Algorithm \ref{alg} outputs behaviors that are aligned with the actual control goal. In our numerical example, the reward signal conveys two requirements: (i) obstructed or forbidden lanes should be avoided whenever possible and (ii) the cars seek to find a free parking space. The first requirement is easy to convey, and we set a very negative ($-20$) reward to such links. In the second case, we exploit the fact that there are two accessible parking lots: we heuristically set the reward to ensure that cars leaving link $l_r$ (see Figure \ref{fig:unisa_map}) follow a policy where the probabilities of going to $l_1$ and $l_2$ are approximately equal. This showcases the practical interest of Algorithm \ref{alg}, as, by merging the sources, we can achieve this. The fleet is then dispatched in a balanced manner, and cars are on average parked sooner.

\end{document}